\documentclass[preprint,3p]{elsarticle}

\usepackage{graphicx}
\usepackage{subcaption}
\usepackage{indentfirst}
\usepackage{ulem}
\usepackage[dvipsnames]{color}
\usepackage{amssymb}
\usepackage{amsmath}
\usepackage{amsthm}
\usepackage[colorlinks=true, citecolor=blue, linkcolor=blue, urlcolor=blue]{hyperref}

\newtheorem{theorem}{Theorem}[section]
\newtheorem{lemma}{Lemma}[section]
\newtheorem{claim}{Claim}[section]
\newtheorem{conjecture}{Conjecture}[section]

\journal{Discrete Applied Mathematics}

\begin{document}

\begin{frontmatter}

\title{The Total Coloring Conjecture holds for planar graphs without  three special subgraphs}

\author[affiliation1]{Rongjin Su}
\affiliation[affiliation1]{organization={School of Computer Science and Cyber Engineering,Guangzhou University},
            city={Guangzhou},
            country={China}}

\author[affiliation2]{Gang Fang}

\author[affiliation2]{Enqiang Zhu\corref{correspondingauthor}}
\cortext[correspondingauthor]{Corresponding author.}
\ead{zhuenqiang@gzhu.edu.cn}
\affiliation[affiliation2]{organization={Institute of Computing Science and Technology,Guangzhou University},
            city={Guangzhou},
            country={China}}

\begin{abstract}

   The Total Coloring Conjecture (TCC) for planar graphs with a maximum degree of six remains open. Previous studies suggest that TCC is valid for such graphs if they do not contain any subgraph isomorphic to a 4-fan. In this paper, we present an improved conclusion by establishing that TCC holds for planar graphs that are free of  three particular substructures, namely the mushroom, the tent, and the cone. This advancement enhances previous findings by demonstrating that TCC is applicable to planar graphs with a maximum degree of six, which can accommodate sparse 4-fans, 5-fans, 5-wheels, and 6-wheels.   
\end{abstract}


\end{frontmatter}


\section{Introduction}\label{sec:introduction}

The graph coloring problem is renowned for its theoretical significance and extensive real-world applications, including the total coloring problem \cite{leidner2012study}. For example, in a judo tournament, multiple matches may occur simultaneously, involving two athletes in each game. Athletes manage their time through three options: they can compete in matches, remain in the arena without participating, or leave the arena during their free periods for rest. In this context, a judo tournament graph can be constructed by representing athletes as vertices and matches as edges connecting these vertices. Two vertices are adjacent if and only if the corresponding athletes are competing against each other in a match. The objective is to minimize the number of periods, which consist of both match periods and free periods, satisfying the following conditions: First,  each athlete may participate in only one match during any given match period; second, athletes competing in an identical match cannot leave the arena in the same free period for safety reasons; third, athletes may only leave the arena during their designated free periods. By coloring the vertices and edges of the judo tournament graph such that no adjacent or incident elements share the same color, this problem can be modeled as the total coloring problem.

Given a graph $G$ with a vertex set $V(G)$ and an edge set $E(G)$, a total $k$-coloring of $G$ is defined as a mapping $\phi$ from $V(G)\cup E(G)$ to a color set $Y=\{1,2,\cdots,k-1,k\}$ of $k$ colors, satisfying that adjacent vertices cannot be assigned the same color, adjacent edges cannot be assigned the same color, and no vertex can share a color with its incident edges. Given an $S\subseteq V(G)\cup E(G)$, a \textit{partial total $k$-coloring regarding $S$} of $G$ is an assignment of colors to elements in $S$ that satisfies the restriction of total coloring. A graph is $totally$ $k$-$colorable$ if it admits a total $k$-coloring. The minimum $k$ for which a graph $G$ is totally $k$-colorable is called the $total$ $chromatic$ $number$ of $G$, denoted by ${\chi''(G)}$.

In the 1960s, Behzad \cite{behzad1965graphs} and Vizing \cite{vizing1968some} independently proposed the Total Coloring Conjecture (TCC), saying that 
\begin{conjecture} \label{tcc-conjecture}
Every graph $G$ has a total $(\Delta(G)+2)$-coloring, i.e. $\chi''(G) \leq (\Delta(G)+2)$, where $\Delta(G)$ is the maximum degree of $G$.
\end{conjecture}

TCC has been verified for graphs with a maximum degree of less than six. The case of graphs with a maximum degree of less than three is straightforward.  Rosenfeld \cite{rosenfeld1971total} and Vijayaditya \cite{vijayaditya1971total} explored the case for graphs with a maximum degree of three, while Kostochka \cite{kostochka1977total} and \cite{kostochka1996total} provided proofs for graphs with maximum degrees of four and five.

A powerful method employed to demonstrate TCC for planar graphs is the well-established discharging technique. This technique involves examining reducible configurations and formulating intricate discharging rules. The TCC has been successfully verified for planar graphs with a maximum degree of seven \cite{sanders1999total}, for those with a maximum degree of eight \cite{yap2006total}, and for graphs with a maximum degree of at least nine \cite{borodin1989total}. Therefore, the only remaining unresolved case of TCC occurs when the maximum degree is equal to six. In this case, the problem becomes more complex. Several researchers have investigated the conjecture's validity under specific conditions. Let $G$ be a planar graph with $\Delta(G) = 6$. Wang et al. \cite{wang2007total} demonstrated that $G$ is totally 8-colorable if it does not contain any 4-cycle. Sun et al. \cite{sun2009total} enhanced this finding by showing that TCC holds for $G$ if it does not contain two triangles sharing a common edge. Roussel \cite{roussel2011local} provided a stronger assertion, indicating that if every vertex $v$ of $G$ is missing some $k_v$-cycle, where $k_v \in \{3,4,5,6,7,8\}$, then TCC is valid. Zhu and Xu \cite{zhu2017sufficient} further refined the results by establishing that TCC is valid if $G$ does not contain any subgraph isomorphic to a 4-fan. For in-depth information on total coloring, please refer to a survey \cite{geetha2023total}.

This paper examines planar graphs with a maximum degree of six that may contain subgraphs isomorphic to a 4-fan. By introducing a novel reducible substructure (outlined in \autoref{sec:newproperties}) and formulating more intricate discharging rules (discussed in \autoref{sec:discharging}), we improve the results presented in \cite{zhu2017sufficient} and establish the following theorem. 

\begin{theorem}\label{thm:theorem1}
    Let $G$ be a planar graph with maximum degree $6$. If $G$ does not contain a subgraph that is isomorphic to a mushroom $($as shown in \autoref{fig:mushroom}$)$, a tent $($as shown in \autoref{fig:tent}$)$, or a cone $($as shown in \autoref{fig:cone}$)$,
    then $G$ is totally $8$-colorable.

    \begin{figure}[ht]\qquad\qquad
        \centering
        \label{fig:condition}
        \begin{subfigure}[b]{0.21\linewidth}
            \centering
            \includegraphics[width=0.8\linewidth]{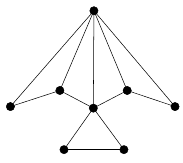}
            \caption{A mushroom}
            \label{fig:mushroom}
        \end{subfigure}\qquad\qquad
        \begin{subfigure}[b]{0.21\linewidth}
            \centering
            \includegraphics[width=0.85\linewidth]{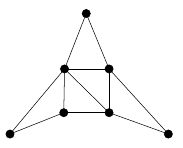}
            \caption{A tent}
            \label{fig:tent}
        \end{subfigure}\qquad\qquad
              \begin{subfigure}[b]{0.13\linewidth}
            \centering
            \includegraphics[width=0.83\linewidth]{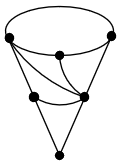}
            \caption{A cone}
            \label{fig:cone}
        \end{subfigure}\qquad\qquad\qquad
        \caption{ Three forbidden graphs}
    \end{figure}
\end{theorem}

\autoref{thm:theorem1} demonstrates that the TCC holds for planar graphs with a maximum degree of six, allowing for sparse 4-fan substructures. This finding builds upon the previous conclusion in \cite{zhu2017sufficient}. Our proof employs the discharging method, based on a minimal counterexample. The structure of the remainder of the paper is as follows. \autoref{sec:preliminaries} introduces the essential notations and terminologies. The proof of \autoref{thm:theorem1} is presented in \autoref{sec:properties}, where in \autoref{sec:reducible}, we examine various structural properties of a minimal counterexample to \autoref{thm:theorem1}, including established properties and new properties related to a specific vertex of degree six, while in \autoref{sec:discharging}, we outline the proof by implementing a series of discharging rules. Following this, \autoref{sec:newproperties} provides a detailed proof of the new properties of a minimal counterexample. Finally, our findings are discussed in \autoref{sec:discussion}.

\section{Preliminaries}\label{sec:preliminaries}

All graphs examined in this paper are simple. By saying a  planar graph, we assume it is embedded in the plane. Let $G$ be a planar graph. We denote the vertex set, edge set, maximum degree, and minimum degree of \( G \) as \( V(G) \), \( E(G) \), \( \Delta(G) \), and \( \delta(G) \), respectively. For a vertex \( v \in V(G) \), a vertex \( u \in V(G) \) is regarded as a \textit{neighbor} of \( v \) if the edge \( uv \) belongs to \( E(G) \). The set of all neighbors of \( v \) is denoted by \( N_G(v) \), and the cardinality of $N_G(v)$ is called the \textit{degree} of vertex \( v \), represented as \( d_G(v) \). A vertex is called a \textit{$k$-, $k^+$-, or $k^-$-vertex} if it has a degree of exactly $k$, at  least $k$, or at  most $k$, respectively. A \textit{\( k \)-neighbor} \( u \) of vertex \( v \) is a neighbor of \( v \) such that \( d_G(u) = k \). We denote the set of faces of \( G \) as \( F(G) \), and the degree of a face \( f \in F(G) \) is the number of edges incident to \( f \) (counting each cut-edge twice), denoted by \( d_G(f) \). A \( k \)-, \( k^+ \)-, \( k^- \)-face is defined as a face with degree exactly \( k \), at least \( k \), and at most \( k \), respectively. A \textit{subgraph} of a graph $G$ is a graph in which the vertex set is a subset of $V(G)$ and the edge set is a subset of $E(G)$. An \textit{induced subgraph by a subset $S\subseteq V(G)$} (or \textit{a subgraph induced by $S\subseteq V(G)$}), denoted by $G[S]$, is a subgraph of $G$ obtained by removing the vertices not in $S$ and their incident edges from $G$. A \textit{$k$-cycle} (denoted as $C_k$) and a \textit{$k$-path} (denoted as $P_k$) refer to cycles and paths of length $k$, respectively. A $3$-cycle is commonly called a \textit{triangle}.
An \textit{$(x,y,z)$-triangle} is defined as a triangle whose three vertices are designated as the $x$-vertex, $y$-vertex, and $z$-vertex.
A \textit{$k$-fan} (denoted as $F_k$) and a \textit{$k$-wheel} (denoted as $W_k$) are specific types of graphs that can be constructed by connecting a vertex $v$ to every vertex in a $k$-path $P_k$ and a $k$-cycle $C_k$, respectively. In both cases, the vertex $v$ is called the \textit{center} of the fan and the wheel.

Let $v$ be  an $\ell$-vertex of a graph $G$. If there are $k$ ($2\leq k\leq \ell$) vertices $y_1,y_2,\ldots, y_k$ in $N_G(v)$ such that $y_iy_{i+1}\in E(G)$ for $i=1,2,\ldots, k-1$ and $y_i$ is a $x_i$-vertex of $G$ for $i=1,2,\ldots, k$, then we refer to $G[\{v, y_1,y_2,\ldots, y_k\}]$ as an \textit{$[x_1,x_2,\ldots,x_k]$ around $v$} and also a \textit{$[y_1,y_2,\ldots,y_k]$ around $v$}. Particularly, if  a $[y_1,y_2,\ldots,y_k]$ around $v$ satisfies that $y_ky_1\in E(G)$, then we call it a \textit{$[y_1,y_2,\ldots,y_k]$ surrounding $v$} and also an \textit{$[x_1,x_2,\ldots,x_k]$ surrounding $v$}. Note that a $[y_1,y_2,\ldots,y_k]$ around $v$ is a $(k-1)$-fan and a $[y_1,y_2,\ldots,y_k]$ surrounding $v$ is a $k$-wheel. Given a $[y_1,y_2,\ldots,y_{k}]$ around $v$,  we refer to $y_{\frac{k+1}{2}}$ (when $k \equiv 1$  (mod 2) ) or both $y_{\frac{k}{2}}$ and $y_{\frac{k}{2}+1}$ (when $k \equiv 0$  (mod 2) ) as the \textit{middle neighbors} of $v$.

Let $\phi$ be a partial total $k$-coloring of $G$ regarding a subset $S$ of $V(G)\cup E(G)$, and $Y= \{1,2,\ldots,k\}$ is the color set of $k$ colors. Given a vertex $v\in V(G)$, we use $C_\phi(v)$ to denote the set of colors that are assigned to the edges incident with $v$, and let $\overline{C_\phi(v)}=Y\setminus C_\phi(v)$,
$C_\phi[v]=C_\phi(v)\cup \{\phi(v)\}$, and $\overline{C_\phi[v]}=Y\setminus C_\phi[v]$. For every $X \subseteq S$, let $\phi(X)=\{\phi(x)| x\in X\}$. For convenience, when $\overline{C_\phi[v]}$ contains exactly one color $c$, i.e. $\overline{C_\phi[v]}=\{c\}$, we simplify it as  $\overline{C_\phi[v]}=c$. In the following, the color set $Y$ is defined as  $Y=\{1,2,3,4,5,6,7,8\}$.

\section{Proof of \autoref{thm:theorem1}} \label{sec:properties}

For all graphs satisfying the conditions of \autoref{thm:theorem1}, we choose a counterexample, denoted by $H$, that minimizes $|V(H)|+|E(H)|$. Then, $H$ satisfies the following four basic properties.

\vspace{0.1cm}

(1) $H$ is a planar graph of maximum degree $6$;

\vspace{0.1cm}

(2) $H$ does not contain a subgraph that is
    isomorphic to a mushroom, a tent, or a cone;

\vspace{0.1cm}
    
(3) $H$  does not have a total 8-coloring;

\vspace{0.1cm}

(4) Every proper subgraph of $H$ is totally 8-colorable. Note that for any proper subgraph (denoted as $Q$) of $H$, if $\Delta(Q)\le5$, we know that $Q$ is totally 8-colorable from the existing results in \cite{kostochka1977total}, \cite{kostochka1996total}, \cite{rosenfeld1971total}, and \cite{vijayaditya1971total};  if $\Delta(Q)=6$, then $Q$ is totally 8-colorable by the minimality of $H$.

\subsection{Reducible configurations} \label{sec:reducible}
We first introduce some substructures that are reducible in $H$; that is, $H$ does not contain these subgraphs. The following five reducible substructures (\autoref{lem3-1} $\sim$ \autoref{lem3-5}) can be determined similarly to those shown in \cite{roussel2011local} and \cite{zhu2017sufficient}. Here, we omit the proofs of them.  

\begin{lemma}\label{lem3-1}
 Let $v\in V(H)$ be an arbitrary vertices. Then, $d_H(v)\geq 3$; in particular, if $d_H(v)=3$, $v$ is not incident with a triangle and $N_H(v)$ contains only $6$-vertices. 
\end{lemma}

\begin{lemma}\label{lem3-2}
Every triangle in $H$ is incident with at most one $4$-vertex. 
\end{lemma}

\begin{lemma}\label{lem3-3}
   $H$ contains no $(4,5,5)$-triangle.
\end{lemma}

\begin{lemma}\label{lem3-4}
 Every $6$-vertex in $H$ has at most four $3$-neighbors.
\end{lemma}

\begin{lemma}\label{lem3-5}
  Let $uvw$ be a $(4,5,6)$-triangle in $H$, where $d_H(u)=4$, $d_H(w)=5$ and $d_H(v)=6$.Then
  
\vspace{0.1cm}
$(i)$ $N_H(v)$ contains no $3$-vertex,
            
\vspace{0.1cm}          
$(ii)$ Edges $uv$ and $uw$ are incident with exactly one triangle $uvw$, and
            
\vspace{0.1cm}
$(iii)$  The $6$-vertex $v$ has at most two $4$-neighbors.
\end{lemma}

In addition to the aforementioned structural properties, we have identified several novel reducible configurations in \( H \) that are vital to the proof of \autoref{thm:theorem1}. Here, we just briefly delineate them, while the comprehensive proofs are presented in \autoref{sec:newproperties}.

\begin{lemma}\label{lem3-6}
Suppose that $H$ contains a $[6,4,6]$ around a $6$-vertex $v$. The corresponding three neighbors of $v$ are $w, u$, and $y$, where $d_H(w)=d_H(y)=6$, $d_H(u)=4$, $wu\in E(H)$, and $uy\in E(H)$. If $v$ is adjacent to a $4$-vertex other than $u$, then the following statements hold.

\vspace{0.1cm}
$(i)$  $N_H(v)$ contains no $3$-vertex.
            
\vspace{0.1cm}          
$(ii)$  $H$  contains no $[4,6,4,6]$ around $v$.
            
\vspace{0.1cm}
$(iii)$  $H$ does not contain a $[4,6,6,4,6,6]$ around $v$.
\end{lemma}

\subsection{Discharging approach}\label{sec:discharging}
We use the discharging method to prove \autoref{thm:theorem1}. Several discharging rules are set according to the properties of $H$. 
By Euler's formula, $|V(H)|-|E(H)|+|F(H)|=2$, we have
    \[
       \sum_{v \in V(H)}(d_H(v)-4) + \sum_{f\in F(H)}(d_H(f)-4) = -8 < 0
    \]
Let each element $x \in V(H) \cup F(H) $ get the  initial charge $ch(x)$: $ch(x)=d_H(x)-4$.
    It satisfies \[\sum_{x \in V(H) \cup F(H)}ch(x)<0 \]
Then, $H$ is discharged according to the following seven rules: R1, R2, R3, R4, R5, R6, R7, which are described below. 
    
    \begin{enumerate}

        \item[R1:] Each 6-vertex gives $\frac{1}{3}$ to every 3-neighbor.
        
        \item[R2:] Each 6-vertex sends $\frac{1}{3}$ to every its incident (6, 5$^+$, 5$^+$)-triangle,
                    $\frac{1}{2}$ to every incident  (4, 6, 6)-triangle,
                   and  $\frac{2}{3}$ to every incident (4, 5, 6)-triangle.
                   
        \item[R3:] Each 5-vertex sends $\frac{1}{3}$ to every its incident 3-face.    
        
        \item[R4:] Each 5-vertex incident with exactly four 3-faces obtains the remaining charges of its middle neighbor after R1 to R3.
        
        \item[R5:] Each 5-vertex incident with exactly five 3-faces obtains the remaining charges of its every neighbor after R1 to R3. 
        
        \item[R6:] Each 6-vertex incident with exactly five 3-faces obtains the remaining charges of its two middle neighbors after R1 to R3.
        
        \item[R7:] Each 6-vertex incident with exactly six 3-faces obtains the remaining charges of its every neighbor after R1 to R3.     
    \end{enumerate}
   
    Let $ch'(x)$ be the final charge of $x \in V(H) \cup F(H) $ after implementing the above discharging procedure on $H$.
    We will prove that $ch'(x)\geq 0$ for each $x \in V(H) \cup F(H)$ and 
    obtain a contradiction. 

    \begin{claim} \label{claim3-1}
     $ch'(f)\geq 0$ for every face $f \in F(H)$.
    \end{claim}
    
 \begin{proof}
  Clearly, $d_H(f)\geq 4$ or $d_H(f)=3$. 
     When $d_H(f)\geq 4$, $f$ does not send any charge to other elements, and hence $ch'(f)=d_H(f)-4 \geq 0$. When $d_H(f)=3$, by \autoref{lem3-1}, \autoref{lem3-2}, \autoref{lem3-3}, we can see that  $f$ is a (4, 5, 6)-triangle,
    a (4, 6, 6)-triangle, or a (5$^+$, 5$^+$, 5$^+$)-triangle. By rules R2 and R3, $ch'(f)=(3-4)+\frac{1}{3}+\frac{2}{3}=0$, 
    $ch'(f)=(3-4)+\frac{1}{2}+\frac{1}{2}=0$, or $ch'(f)=(3-4)+\frac{1}{3}+\frac{1}{3}+\frac{1}{3}=0$, respectively.
    \end{proof}

Now, we compute the final  charges of vertices. Let $v\in V(H)$ be an arbitrary vertex.  By \autoref{lem3-1}, we have $3\leq d_H(v)\leq 6$. When $d_H(v)=3$, $v$ has three 6-neighbors by \autoref{lem3-1} and  $ch'(v)=(3-4)+\frac{1}{3}+\frac{1}{3}+\frac{1}{3}=0$ by rule R1. When $d_H(v)=4$, according to the discharging rules, $v$ does not send any charge to other vertices or faces, hence $ch'(v)=4-4=0$. In the following, we consider the case of $5\leq d_H(v)\leq 6$. 
    
\begin{claim} \label{claim3-2}
$ch'(v)\geq 0$ for every $5$-vertex $v \in V(H)$.
\end{claim}
\begin{proof}  
Let $k$ be the number of 3-faces that are incident with $v$. We have $k\leq 5$.  When $k\leq 3$, according to the rule R3,  $ch'(v)\geq (5-4)-\frac{1}{3}\times 3=0$. 

When $k=4$, there is a $[v_1,v_2,v_3,v_4,v_5]$ around $v$, where $N_H(v)=\{v_1,v_2,v_3,v_4,v_5\}$ and $v_3$ is the middle neighbor of $v$. By \autoref{lem3-5} (ii), we have $d_H(v_3)\geq 5$, $d_H(v_2)\geq 5$, and $d_H(v_4)\geq 5$. By property (2) of $H$, $v_3$ is incident with at most two 3-faces. If $d_H(v_3)= 5$, according to the rules R3 and R4, we have $ch'(v)\geq(5-4)-\frac{1}{3}\times 4+(5-4-\frac{1}{3}\times 2) =0$. If $d_H(v_3)= 6$, according to the rules R1, R2, R3 and R4, we have $ch'(v)\geq(5-4)-\frac{1}{3}\times 4+(6-4-\frac{1}{3}\times 2-\frac{1}{3}\times 3) =0$. 

When $k=5$, there is a 5-wheel with center $v$. 
By \autoref{lem3-3} and \autoref{lem3-5} (ii), $v$ has no 4$^-$-neighbor, and every neighbor of $v$ is incident with at most two 3-faces by property (2) of $H$. Then,  according to the rules R1, R2, R3 and R5, we have $ch'(v)\geq 1-\frac{1}{3}\times 5+5\times \min\{(1-\frac{1}{3}\times 2), \enspace (2-\frac{1}{3}\times 2-\frac{1}{3}\times 3) \} =1$. 
\end{proof}

\begin{claim} \label{claim3-3}
$ch'(v)\geq 0$ for every $6$-vertex $v \in V(H)$.
\end{claim}

\begin{proof}
Consider the vertex \( v \) with exactly \( t \) 3-neighbors. By \autoref{lem3-1} and \autoref{lem3-2}, $v$ is not incident with a  ($3^-$,$6^-$,$6^-$)-triangle or a ($4^-$,$4^-$,$6^-$)-triangle. According to \autoref{lem3-4}, it follows that \( t \leq 4 \). Importantly, if \( t \geq 1 \), then vertex \( v \) cannot be part of a (4,5,6)-triangle, as indicated in \autoref{lem3-5} (i). Moreover, if \( t \geq 1 \), by \autoref{lem3-5} (ii), there cannot be a [6,4,5] around \( v \). Additionally, \autoref{lem3-6} (i) indicates that if there is a [6,4,6] around $v$, then $v$ has at most one 3-neighbor; 
if there is a [6,4,6] around $v$ and $v$ has a 3-neighbor, then $v$ has four 5$^+$-neighbors. 
\autoref{lem3-6} (ii) establishes that there cannot be a [6,4,6,4]  around  \( v \). In the subsequent analysis, we will determine \( ch'(v) \) by considering five different cases based on the value of \( t \).

When $t=4$, since no 3-vertex is incident with a triangle by \autoref{lem3-1}, $v$ is incident with at most one triangle. If $v$ is indeed incident with a triangle, it can only be a (4,6,6)-triangle or (5$^+$,5$^+$,6)-triangle. Following the guidelines outlined in rules R1 and R2, we can deduce that  $ch'(v)\geq(6-4)-\frac{1}{3}\times 4-\max\{\frac{1}{2},\enspace \frac{1}{3}\}=\frac{1}{6}$. 

When $t=3$, $v$ is incident with at most two triangles, and the neighborhood $N_H[v]$ of $v$ can only form a [4,6,4], a [4,6,5$^+$] or a [5$^+$,5$^+$,5$^+$] around $v$. According to the rules R1 and R2, we have $ch'(v)\geq(6-4)-\frac{1}{3}\times 3-\max\{(\frac{1}{2}\times 2),\enspace (\frac{1}{2}+\frac{1}{3}),\enspace (\frac{1}{3}\times 2)\}=0$.

When $t=2$, $v$ is incident with at most three triangles, and $N_H[v]$ can only form a [4,6,6,4], a [4,6,5$^+$,5$^+$] or a [5$^+$,5$^+$,5$^+$,5$^+$] around $v$. According to the rules R1 and R2, we have $ch'(v)\geq(6-4)-\frac{1}{3}\times 2-\max\{(\frac{1}{2}\times 2+\frac{1}{3}),\enspace (\frac{1}{3}\times 2+\frac{1}{2}),\enspace (\frac{1}{3}\times 3)\}=0$.

When $t=1$,  $v$ is incident with at most four triangles, and  $N_H[v]$ can only form a [6,4,6,5$^+$,5$^+$], a [5$^+$,6,4,6,5$^+$], a [4,6,5$^+$,6,4], a [4,6,5$^+$,5$^+$,5$^+$], or a [5$^+$,5$^+$,5$^+$,5$^+$,5$^+$] around $v$.
According to the rules R1 and R2, we have $ch'(v)\geq(6-4)-\frac{1}{3}-\max\{(\frac{1}{2}\times 2+\frac{1}{3}\times 2),\enspace (\frac{1}{2}\times 2+\frac{1}{3}\times 2),\enspace(\frac{1}{2}\times 2+\frac{1}{3}\times 2),\enspace (\frac{1}{2}+\frac{1}{3}\times 3),\enspace(\frac{1}{3}\times 4)\}=0$.

In the following discussion, we focus on the case where \( t = 0 \). According to \autoref{lem3-5} (iii), when vertex \( v \) is incident with a (4,5,6)-triangle, it can have at most two 4-neighbors. If \( v \) has no 4-neighbor, then $N_H[v]$ can only form a \([5^+, 5^+, 5^+, 5^+, 5^+, 5^+]\) surrounding $v$. Based on rules R2, R6, and R7, we can conclude that \( ch'(v) \geq (6 - 4) - \frac{1}{3} \times 6 = 0 \). In the case where \( v \) does have a 4-neighbor, we will analyze three additional cases depending on the number of 3-faces that are incident with \( v \).

\textit{Case 1.} $v$ is incident with at most four 3-faces. 

When there is a [6,4,6] around $v$, $N_H[v]$ can only form a [6,4,6,5,4], a [6,4,6,6,4], a [6,4,6,5$^+$,5$^+$], or a [5$^+$,6,4,6,5$^+$] around $v$.
According to the rule R2,  we have $ch'(v)\geq(6-4)-\max\{(\frac{1}{2}\times 2+\frac{1}{3}+\frac{2}{3}),\enspace (\frac{1}{2}\times 3+\frac{1}{3}),\enspace (\frac{1}{2}\times 2+\frac{1}{3}\times 2),\enspace (\frac{1}{2}\times 2+\frac{1}{3}\times 2)\}=0$.
 
When there is not a [6,4,6] around $v$, $N_H[v]$ can only form a  [4,5,5$^+$,5,4], a [4,6,5$^+$,6,4], a [4,5,5$^+$,6,4], a [4,5,5$^+$,5$^+$,5$^+$], a [4,6,5$^+$,5$^+$,5$^+$], or a [5$^+$,5$^+$,5$^+$,5$^+$,5$^+$] around $v$. 
According to the rule R2, we have $ch'(v)\geq(6-4)-\max\{(\frac{2}{3}\times 2+\frac{1}{3}\times 2),\enspace (\frac{1}{2}\times 2+\frac{1}{3}\times 2),\enspace (\frac{2}{3}+\frac{1}{3}\times 2+\frac{1}{2}),\enspace (\frac{2}{3}+\frac{1}{3}\times 3),\enspace (\frac{1}{2}+\frac{1}{3}\times 3),\enspace (\frac{1}{3}\times 4)\}=0$.

\textit{Case 2.} $v$ is incident with exactly five 3-faces. 
In this case,  $v$ is the center of a 5-fan. Let  $u_3$ and $u_4$ be the middle neighbors of $v$. By property (2) of $H$, $u_3$ (and $u_4$) is incident with at most two 3-faces by property (2) of $H$. In addition, 
$v$ has at most two 4-neighbors by \autoref{lem3-6} (ii).

When $v$ has exactly one 4-neighbor,  $N_H[v]$ can only form a  [4,5,5$^+$,5$^+$,5$^+$,5$^+$], a [4,6,5$^+$,5$^+$,5$^+$,5$^+$], a [6,4,6,5$^+$,5$^+$,5$^+$], or a [5$^+$,6,4,6,5$^+$,5$^+$] around $v$.  
According to the rules R2 and R6, we have $ch'(v)\geq(6-4)-\max\{(\frac{2}{3}+\frac{1}{3}\times 4),\enspace (\frac{1}{2}+\frac{1}{3}\times 4),\enspace (\frac{1}{2}\times 2+\frac{1}{3}\times 3),\enspace (\frac{1}{2}\times 2+\frac{1}{3}\times 3)\}=0$.

When $v$ has exactly two 4-neighbors, since $H$ contains no [4,6$^-$,6$^-$,4,6$^-$,6$^-$] around $v$ by \autoref{lem3-6} (iii),
$N_H[v]$ can only form a [6,4,6,6,4,6], a [4,5$^+$,5$^+$,5$^+$,5$^+$,4], or a [6,4,6,5$^+$,5$^+$,4] around $v$.
According to the rules R1, R2, R3 and R6, 
    if there is a [6,4,6,6,4,6] around $v$, we have $ch'(v)\geq (2-\frac{1}{2}\times 4-\frac{1}{3})+(2-\frac{1}{3}\times 3-\frac{1}{2}-\frac{1}{3})\times 2 =0$; if there is a [4,5$^+$,5$^+$,5$^+$,5$^+$,4] around $v$, 
    we have $ch'(v)\geq (2-\max\{\frac{2}{3},\enspace \frac{1}{2}\}\times 2-\frac{1}{3}\times 3)+ \min\{(1-\frac{1}{3}\times 2),\enspace (2-\frac{1}{3}\times 2-\frac{1}{3}\times 3)\}\times 2 =\frac{1}{3}$; if there is a [6,4,6,5$^+$,5$^+$,4] around $v$, 
    we have $ch'(v)\geq (2-\frac{1}{2}\times 2-\frac{1}{3}\times 2-\max\{\frac{2}{3},\enspace \frac{1}{2}\})+ (2-\frac{1}{2}-\frac{1}{3}-\frac{1}{3}\times 3)+ \min\{(1-\frac{1}{3}\times 2),\enspace (2-\frac{1}{3}\times 2-\frac{1}{3}\times 3)\}=\frac{1}{6}$.

\textit{Case 3.} $v$ is incident with six 3-faces.  Then, $v$ is the center of a 6-wheel, and $N_H[v]$ can only form  a [6,4,6,5$^+$,5$^+$,5$^+$] surrounding $v$ by \autoref{lem3-6} (ii) and \autoref{lem3-6} (iii). According to the rules R1, R2, R3 and R7, we have $ch'(v)\geq (2-\frac{1}{2}\times 2-\frac{1}{3}\times 4) + (2-\frac{1}{2}-\frac{1}{3}-\frac{1}{3}\times 3)\times 2 + 3\times \min\{(2-\frac{1}{3}\times 2-\frac{1}{3}\times 3),\enspace (1-\frac{1}{3}\times 2)\} =1$.
\end{proof}

Based on \autoref{claim3-1}, \autoref{claim3-2}, and \autoref{claim3-3}, we establish that $\sum_{x \in V(H) \cup F(H)} ch'(x) \geq 0$, which results in a contradiction. Consequently, there exists no minimal counterexample to \autoref{thm:theorem1}. This concludes the demonstration of \autoref{thm:theorem1}.

\section{Proof of \autoref{lem3-6}}\label{sec:newproperties}

In this section, we will present a proof for \autoref{lem3-6}. The statement of the lemma aligns with what is outlined in \autoref{lem3-6}. We will begin by demonstrating that the graph \( H \) can be assigned a partial total 8-coloring regarding the set \( (V(H) \setminus \{u\}) \cup E(H) \). Following this, we will establish the lemma by showing that this partial total 8-coloring can be expanded into a total 8-coloring of \( H \) by appropriately coloring \( u \), if \( H \) containing the configurations described in \autoref{lem3-6}. 

 Suppose that $H$ contains a $[6,4,6]$ around a $6$-vertex $v$. The corresponding three neighbors of $v$ are $w, u$, and $y$, where $d_H(w)=d_H(y)=6$, $d_H(u)=4$, $wu\in E(H)$, and $uy\in E(H)$.

\begin{claim} \label{claim4-1}
$H$ has a partial total 8-coloring regarding  $(V(H)\setminus\{u\})\cup E(H)$.  
\end{claim}
\begin{proof}
According to the property (4) of graph \( H \), it possesses a partial total 8-coloring \( h \) regarding $V(H)\cup (E(H)\backslash\{uv\})$. We aim to modify the coloring to appropriately color the edge \( uv \) after removing the color assigned to vertex \( u \).
Let \( N_H(u) = \{w,v,y,x\} \). Without loss of generality, we can assume that \( C_h[v] = \{1,2,3,4,5,6\} \), where \( h(v) = 1 \), \( h(wv) = 6 \), and \( h(vy) = 2 \). 
We give the proof by contradiction.
Assuming that there is no available color for the edge \( uv \).  
Note that if $7$ or $8\notin\{h(wu),h(uy),h(xu)\}$, then we can color the edge $uv$ with 7 or 8 to obtain a partial total 8-coloring regarding $(V(H)\setminus\{u\})\cup E(H)$.
Therefore we have \(\{7,8\}\subseteq \{h(wu),h(uy),h(xu)\}\), \(  C_h[v] \cup C_h(u) = Y =\{1,2,3,4,5,6,7,8\}\),  and  \( \{h(wu), h(uy)\} \cap \{7,8\} \neq \emptyset \). We can see four  possibilities: $h(uy)=7$, $h(uy)=8$, $h(wu)=7$, or $h(wu)=8$. It is sufficient to consider any one of the possibilities due to symmetry or similarity. For simplicity, let’s say \( h(uy) = 7 \), \( h(xu) = a \), and \( h(wu) = b \) (as shown in \autoref{fig:Figure_2}). 
Note that $|\overline{C_h[y]}|=|\overline{C_h[w]}|=1$ and $|\overline{C_h[v]}|=2$.  Let $\overline{C_h[y]}=c$ and $\overline{C_h[w]}=c'$.
We will examine two cases: one in which \( a = 8 \) and another in which \( b = 8 \).

    \begin{figure}[htbp]
        \centering
        \begin{minipage}{0.22\linewidth}
            \centering
            \includegraphics[width=\textwidth]{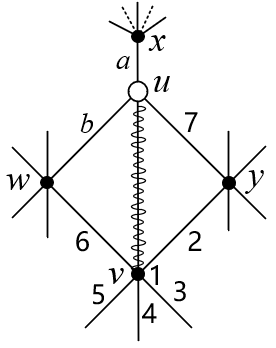}
            \caption{}
            \label{fig:Figure_2}
        \end{minipage}\qquad\qquad\qquad\qquad\qquad
        \begin{minipage}{0.22\linewidth}
            \centering
            \includegraphics[width=\textwidth]{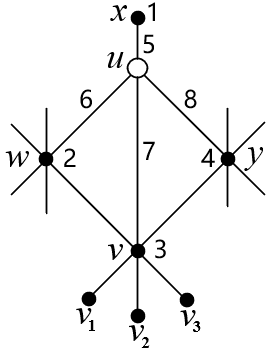}
            \caption{}
            \label{fig:Figure_3}
        \end{minipage}\qquad\quad   
    \end{figure}
  
Regarding $a=8$, we consider three cases: $c\notin \{8,b\}$, $c=b$, and $c=8$. First, $c\notin \{8,b\}$. We can recolor $uy$ with $c$ and color $uv$ with 7. Second, $c=b$. When $c'\notin \{7,8\}$, we can recolor $wu$ with $c'$, recolor $uy$ with $b$, and color $uv$ with 7. When $c'\in \{7,8\}$, we can recolor $wv$ with $c'$ and color $uv$ with 6. Third, $c=8$. When $b\neq 2$, we can recolor $vy$ with 8 and color $uv$ with 2. When $b=2$, we have $c' \in \{1,3,4,5,7,8\}$: If $c'\in \{1,3,4,5\}$, we can recolor $wu$ with $c'$, recolor $vy$ with 8, and color $uv$ with 2; if $c'\in \{7,8\}$, we can recolor $wv$ with $\overline{C_h[w]}$, and color $uv$ with 6.

Regarding $b=8$, we consider three cases: $c'\notin \{7,a\}$, $c'=7$, and $c'=a$. 
First, $c'\notin \{7,a\}$. We can recolor $wu$ with $c'$ and color $uv$ with 8. Second, $c'=7$.  When $a\neq 6$, we can recolor $wv$ with 7 and color $uv$ with 6. When $a=6$, if $c\in\{1,3,4,5\}$, we can recolor $uy$ with $c$ and recolor $uv$ with 7; if $c\in\{6,8\}$, we can recolor $vy$ with $c$, recolor $wv$ with 7 and color $uv$ with 2. Third, $c'=a$. In this case, we have that $a\in \{1,2,3,4,5\}$ and there is a color
$c_2\in \overline{C_h[x]}$ such that $c_2\in Y\setminus \{a\}$.
We recolor the edge $xu$ with $c_2$. Moreover, if $c_2\notin \{7,8\}$, we recolor $wu$ with $a$ and color $uv$ with 8. If $c_2=8$, we recolor $wu$ with $a$, recolor $wv$ with 8, and color $uv$ with 6. If $c_2=7$, we have $c\in \{1,3,4,5,6,8\}$.
When $c=6$, we can recolor $wu$ with $a$, recolor $uy$ with 6, and color $uv$ with 8.
When $c=8$, we can  recolor $wu$ with $a$, the edge $uy$ with 8, $wv$ with 8, and color $uv$ with 6.
When $c\in \{1,3,4,5\}\backslash\{a\}$, we can recolor $wu$ with $a$, recolor $uy$ with $\overline{C_h[y]}$, and color $uv$ with 8.
Finally, when $\overline{C_h[y]}=a$, we can recolor $uy$ with $a$, recolor $vy$ with 7, and color $uv$ with 2.
    So we obtain the partial total 8-coloring regarding $(V(H)\backslash\{u\})\cup E(H)$.  
\end{proof}

Let \( g \) be a partial total 8-coloring regarding \( (V(H) \setminus \{u\}) \cup E(H) \).  According to property (3) of \( H \), the partial total 8-coloring \( g \) cannot be extended to form a total 8-coloring for the entire graph \( H \). In each configuration presented in \autoref{lem3-6}, if we can assign a color to the vertex \( u \) in such a way that results in a total 8-coloring of H, this would lead to a contradiction. This contradiction indicates that the configuration can be excluded from \( H \).

 Let \( N_H(v) = \{u, w, y, v_1, v_2, v_3\} \). If \( g(N_H(u)) \cup C_g(u) \neq Y \), there is at least one available color in \( Y \) for vertex \( u \). This implies that \( H \) has a total 8-coloring, which leads to a contradiction. Therefore, we must examine the case where \( g(N_H(u)) \cup C_g(u) = Y \). Without loss of generality, we can assume the following color assignments: \( g(v) = 3 \), \( g(w) = 2 \), \( g(y) = 4 \), \( g(x) = 1 \), \( g(uv) = 7 \), \( g(wu) = 6 \), \( g(uy) = 8 \), and \( g(xu) = 5 \) (as illustrated in \autoref{fig:Figure_3}). We observe that \( \{1, 2, 3, 4\} \subseteq (C_g[v] \cap C_g[w] \cap C_g[y]) \). If this is not the case, there exists a color \( c_3 \in \{1, 2, 3, 4\} \) and a vertex \( z \in \{v, w, y\} \) such that \( c_3 \notin C_g[z] \). In this situation, we can assign the color \( g(uz) \) to vertex \( u \) and recolor the edge \( uz \) with \( c_3 \) , which would also result in a contradiction. Let $\overline{C_g[w]}=c'$ and $\overline{C_g[y]}=c$. We conclude that \( \overline{C_g[v]} \in \{5, 6, 8\} \), \( c' \in \{5, 7, 8\} \), and \( c\in \{5, 6, 7\} \). The following result (\autoref{claim4-2} and \autoref{claim4-3}) determines the range of \( \overline{C_g[v]}\) and \(g(wv) \) when  \( v \) has a 4-neighbor distinct from \( u \).

 \begin{claim}\label{claim4-2}
    If $N_H(v)\setminus \{u\}$ contains a $4$-vertex, say $v_1$, then  \( \overline{C_g[v]} = 5 \).
 \end{claim}

 \begin{proof}
    Suppose that $\overline{C_g[v]} \neq  5$. Then,  $\overline{C_g[v]}\in\{6,8\}= \{g(wu),g(uy)\}$. We will demonstrate that $g$ can be extended to a total 8-coloring of $H$, leading to a contradiction.
    By symmetry or similarity, we only analyze the case where $\overline{C_g[v]}=g(uy)$, which implies $\overline{C_g[v]}=8$. We can observe that $g(vy)\in \{1,2,5,6\}$. 
    
    First, if $g(vy)\in \{1,2\}$, then we can recolor the edge $uy$ with $g(vy)$, recolor $vy$ with 8, and color the vertex $u$ with 8. 
    
    Second, if $g(vy)=5$, we have that $c'\in \{7,8,5\}$, $g(wv)\in \{1,4\}$, and $c\in \{6,7\}$. Furthermore, if $c'=7$, we can recolor $uv$ with $g(wv)$, recolor $wv$ with 7, and color $u$ with 7. If $c'=8$, we can recolor $wu$ with $g(wv)$, recolor $wv$ with 8, and color $u$ with 6.  If $c'=5$, when $c=6$, we can recolor $wu$ with $g(wv)$, the edge $wv$ with 5, $vy$ with 8, $uy$ with 6, and color $u$ with 8; when  $c=7$, we can recolor $uv$ with $g(wv)$, the edge $wv$ with 5, $vy$ with 7, and color $u$ with 7. 
    
    Third, if $g(vy)=6$, then $g(wv)\in \{1,4,5\}$, and $c\in \{5,7\}$.
    When $g(wv)\in \{1,4\}$, we can recolor $wu$ with $g(wv)$, the edge $wv$ with 6, $uy$ with 6, $vy$ with 8, and color $u$ with 8. 
    When $g(wv)=5$, we have $c'\in \{7,8\}$, and $\{g(vv_1),g(vv_2),g(vv_3)\}=\{1,2,4\}$. 
    
    Now, we show that $\{g(v_1),g(v_2),g(v_3)\}=\{5,8,6\}$ and $C_g[v_1]=\{g(vv_1),5,7,8,6\}$ (see \autoref{fig:Figure_4}). Otherwise, if $5\notin \{g(v_1),g(v_2),g(v_3)\}$, we can recolor the vertex $v$ with 5, the edge $uv$ with 3, $wv$ with $c'$, and color $u$ with 7; if $8\notin \{g(v_1),g(v_2),g(v_3)\}$, we can recolor the vertex $v$ with 8 and color $u$ with 3; if $6\notin \{g(v_1),g(v_2),g(v_3)\}$, we can recolor the vertex $v$ with 6, the edge $vy$ with $c$, $uv$ with 3, and color $u$ with 7 (We need to recolor $wv$ with $c'$ when $c=5$). Moreover, if $5\notin C_g[v_1]$, we can recolor $uv$ with $g(vv_1)$, the edge $vv_1$ with 5, $wv$ with $\overline{C_g[w]}$, and color $u$ with 7; if $7\notin C_g[v_1]$, we can recolor $uv$ with $g(vv_1)$, recolor $vv_1$ with 7, and color $u$ with 7; if $8\notin C_g[v_1]$, we can recolor $uv$ with $g(vv_1)$, recolor $vv_1$ with 8, and color $u$ with 7; if $6\notin C_g[v_1]$, we can recolor $uv$ with $g(vv_1)$, the edge $vv_1$ with 6, $vy$ with $\overline{C_g[y]}$, and color $u$ with 7 (We need to recolor $wv$ with $c'$ when $c=5$). 
    The above discussion indicates that  $7\notin \{g(v_1),g(v_2),g(v_3)\}$ and $3\notin C_g[v_1]$. Then, we can recolor $uv$ with $g(vv_1)$, recolor $vv_1$ with 3, recolor $v$ with 7, and color $u$ with 3.

    In either case, extending the \( g \) to a total 8-coloring of \( H \) is always possible, thereby resulting in a contradiction.
    \end{proof}

        \begin{figure}[htbp]
            \centering
            \qquad
            \begin{minipage}{0.21\linewidth}
                \centering
                \includegraphics[width=\textwidth]{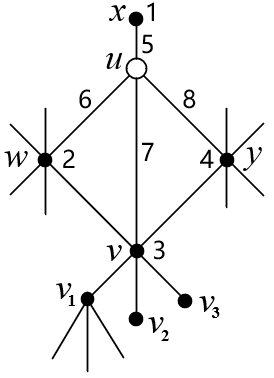}
                \caption{}
                \label{fig:Figure_4}
            \end{minipage}\qquad\qquad\qquad\qquad\qquad\qquad
            \begin{minipage}{0.21\linewidth}
                \centering
                \includegraphics[width=\textwidth]{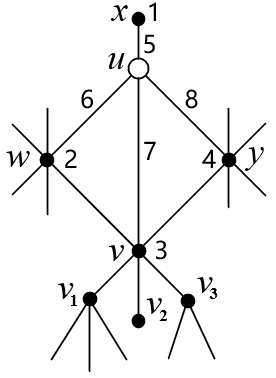}
                \caption{}
                \label{fig:Figure_5}
            \end{minipage}\qquad\qquad
        \end{figure}

 \begin{claim}\label{claim4-3}
    If $N_H(v)\setminus \{u\}$ contains a $4$-vertex, say $v_1$, then  \( g(wv) \in \{1, 4\} \).
 \end{claim}

\begin{proof}
    By \autoref{claim4-2}, it follows that $\overline{C_g[v]}=5$. It is obvious that $g(wv)\in\{1,4,8\}$. We will show that $g$ can be extended to a total 8-coloring of $H$ if $g(wv)=8$ and obtain a contradiction. Observe that $5\in \{g(v_1),g(v_2),g(v_3)\}$; otherwise, if $5\notin \{g(v_1),g(v_2),g(v_3)\}$, $g$ can be extended to a total 8-coloring of $H$ by recoloring the vertex $v$ with 5 and coloring $u$ with 3, a contradiction. Now, we suppose $g(wv)=8$ and proceed by analyzing the value of $g(vy)$. It is easy to see that  $g(vy)\in\{1,2,6\}$.

    First, if $g(vy)\in \{1,2\}$, then $c\in\{5,7,6\}$, and $c'\in\{5,7\}$.   When $c\in\{5,7\}$, we can recolor the edge $uv$ with $g(vy)$, recolor $vy$ with $\overline{C_g[y]}$, and color $u$ with 7. 
    When $c=6$, we can recolor $wv$ with $\overline{C_g[w]}$, the edge $uv$ with $g(vy)$, $vy$ with 8, $uy$ with 6, $wu$ with 8, and color $u$ with 7.

    Second, if $g(vy)=6$, then we have $\{g(vv_1),g(vv_2),g(vv_3)\}=\{1,2,4\}$, $c\in\{5,7\}$, and $c'\in\{5,7\}$. 
    We deduce that $\{g(v_1),g(v_2),g(v_3)\}=\{5,6,8\}$ and $C_g[v_1]=\{g(vv_1),5,7,6,8\}$.
    Otherwise,  if $6\notin \{g(v_1),g(v_2),g(v_3)\}$, we can recolor the vertex $v$ with 6, recolor the edge $uv$ with 3, $vy$ with  $\overline{C_g(y)}$, and color $u$ with 7; 
    if $8\notin \{g(v_1),g(v_2),g(v_3)\}$, we can recolor the vertex $v$ with 8, recolor the edge $uv$ with 3, $wv$ with  $\overline{C_g(w)}$, and color $u$ with 7. Moreover, if $5\notin C_g[v_1]$ or $7\notin C_g[v_1]$, we can recolor $uv$ with $g(vv_1)$, recolor $vv_1$ with 5 or 7, and color $u$ with 7;
    if  $6\notin C_g[v_1]$, we can recolor $uv$ with $g(vv_1)$, the edge $vv_1$ with 6, $vy$ with $\overline{C_g(y)}$, and color $u$ with 7;
    if  $8\notin C_g[v_1]$, we can recolor $uv$ with $g(vv_1)$, the edge $vv_1$ with 8, $wv$ with $\overline{C_g(w)}$, and color $u$ with 7. 

The analysis above demonstrates that $7\notin \{g(v_1),g(v_2),g(v_3)\}$ and $3\notin C_g[v_1]$. Then, $g$ can be extended to a total 8-coloring of $H$ by the following steps: recolor $uv$ with $g(vv_1)$, recolor $vv_1$ with 3, recolor $v$ with 7, and color $u$ with 3.  This completes the proof of the claim. 
\end{proof}

By \autoref{claim4-2} and \autoref{claim4-3}, we have that 
$\overline{C_g[v]}=5$ and  $g(wv)\in\{1,4\}$. In the following, we determine the value of $\overline{C_g[w]}$ and $\overline{C_g[y]}$. 

\begin{claim} \label{claim4-4}
 $c'=8$ and $c=6$. 
\end{claim}

\begin{proof}
Note that $c'\in  \{5,7,8\} $ and $c\in \{5,6,7\}$.
If $c'\in\{5,7\}$, we can recolor $uv$ with $g(wv)$, recolor $wv$ with $\overline{C_g[w]}$, and color $u$ with 7. This extends $g$ to a total 8-coloring of $H$, a contradiction. Therefore, $c'=8$. Then, 
we have $g(vy)\in \{1,2,6\}$.

Additionally,  when $c\in\{5,7\}$, if $g(vy)\in\{1,2\}$, we can recolor $uv$ with $g(vy)$, recolor $vy$ with $c$, and color the vertex $u$ with 7; if $g(vy)=6$, we can recolor $uv$ with $g(wv)$, recolor both $wv$ and $uy$ with 6, recolor $wu$ with 8, recolor $vy$ with $c$, and color the vertex $u$ with 7. This also extends $g$ to a total 8-coloring of $H$, a contradiction. Consequently, we have $c=6$. 
\end{proof}

According to \autoref{claim4-2}, \autoref{claim4-3}, and \autoref{claim4-4}, we have that  $\overline{C_g[v]}=5$, $\overline{C_g[w]}=8$, $\overline{C_g[y]}=6$, and $g(wv)\in\{1,4\}$. Note that $\overline{C_g[y]}=6$ also implies that $g(vy)\in\{1,2\}$. 

\begin{claim} \label{claim4-5}
Let $z \in\{v_1,v_2,v_3\}$. Then, $\{5,7\}\subseteq C_g[z]$, and when $g(vz)\in\{6,8\}$, 
 $\{g(wv),g(vy),5,7,g(vz)\}\subseteq C_g[z]$.
\end{claim}
\begin{proof}
It is clear to see that $\{g(vv_1),g(vv_2),g(vv_3),g(vy),g(wv)\}=\{1,2,4,6,8\}$. For any $z \in\{v_1,v_2,v_3\}$, we deal with two cases: $g(vz)\in\{1,2,4\}$ and $g(vz)\in\{6,8\}$.  
First, $g(vz)\in\{1,2,4\}$.  If $5\notin C_g[z]$ or $7\notin C_g[z]$, we can recolor $uv$ with $g(vz)$, recolor $vz$ with 5 or 7, and color $u$ with 7. Second, $g(vz)\in\{6,8\}$. In this case, if $5\notin C_g[z]$, we can recolor $uy$ with $g(vy)$, the edge $vy$ with $g(vz)$, $vz$ with 5, and color $u$ with 8. If $7\notin C_g[z]$, when $g(vz)=6$, we can recolor $uv$ with $g(vy)$, the edge $vy$ with 6, $vz$ with 7, and color $u$ with 7; when $g(vz)=8$,  we can recolor $uv$ with $g(wv)$, the edge $wv$ with 8, $vz$ with 7, and color $u$ with 7. In either case, we can obtain a total 8-coloring of $H$ based on $g$, a contradiction. Therefore, $\{5,7\}\subseteq (C_g[v_1]\cap C_g[v_2]\cap C_g[v_3])$. 

Additionally, when $g(vz)\in\{6,8\}$, if $g(wv)\notin C_g[z]$, we can recolor both $vz$ and $wu$ with $g(wv)$, recolor $wv$ with $g(vz)$, and color $u$ with 6; if $g(vy)\notin C_g[z]$, we can recolor both $vz$ and $uy$ with $g(vy)$, recolor $vy$ with $g(vz)$, and color $u$ with 8. This also obtains a total 8-coloring of $H$ based on $g$, a contradiction. Therefore, $\{g(wv),g(vy),5,7,g(vz)\}\subseteq C_g[z]$. 
\end{proof}

Now, we turn to proving the three statements in \autoref{lem3-6} based on \autoref{claim4-2}, \autoref{claim4-3}, \autoref{claim4-4}, and \autoref{claim4-5}.

\textit{Proof of \autoref{lem3-6} (i).} Suppose to the contrary that $v$ has a 3-neighbor, say $v_3$ (see \autoref{fig:Figure_5}). Observe that $\overline{C_g[v]}=5$, $\overline{C_g[w]}=8$, $\overline{C_g[y]}=6$,
$g(wv)\in \{1,4\}$, $g(vy)\in \{1,2\}$, $\{g(wv),g(vy),g(vv_1), g(vv_2),g(vv_3)\}=\{1,2,4,6,8\}$, 
$\{5,7\}\subseteq C_g[v_1]$, and $\{5,7\}\subseteq C_g[v_3]$. 
We erase the color of $v_3$. Since $d_H(v_3)=3$ and $|g(N_H(v_3)) \cup C_g(v_3)|\le 6<8$, there is always an available color for $v_3$ to be recolored finally. So, we can ignore the color of $v_3$. 
 
We consider two cases based on the value of $g(vv_3)$: $g(vv_3)\in\{6,8\}$ and $g(vv_3)\in\{1,2,4\}$.
If $g(vv_3)\in\{6,8\}$, then we have that 
$\{g(wv),g(vy),5,7,g(vv_3)\} \subseteq C_g[v_3]$, which contradicts $d_H(v_3)=3$. Therefore, $g(vv_3)\in\{1,2,4\}$. 

If $5\notin C_g(v_3)$ or $7\notin C_g(v_3)$, we can recolor $uv$ with $g(vv_3)$, recolor $vv_3$ with 5 or 7, and color $u$ with 7. This contradicts the property (3) of $H$.  Therefore,  $C_g(v_3)=\{g(vv_3),5,7\}$. Then, $g(vv_1)\in\{6,8\}$, and $\{g(wv),g(vy),5,7,g(vv_1)\}\subseteq C_g[v_1]$, i.e. $C_g[v_1]=\{g(wv),g(vy),5,7,g(vv_1)\}$.  We can see that $g(vv_3)\notin C_g[v_1]$ and $g(vy)\notin C_g(v_3)$. As a result, we obtain a total 8-coloring of $H$ by the following steps: recolor both $vv_3$ and $uy$ with $g(vy)$, recolor $vy$ with $g(vv_1)$, recolor $vv_1$ with $g(vv_3)$, and color $u$ with 8, a contradiction. This completes the proof of \autoref{lem3-6}(i).

\textit{Proof of \autoref{lem3-6} (ii).}  Suppose to the contrary that $H$ contains a [4,6,4,6] around the 6-vertex $v$, say $[v_1,w,u,y]$ around $v$, where $d_H(w)=d_H(y)=d_H(v)=6$ and $d_H(v_1)=d_H(u)=4$. We denote $v_1$ by $t$ (see \autoref{fig:Figure_6}). By the discussion above, we have that $\overline{C_g[v]}=5$, $\overline{C_g[w]}=8$, $\overline{C_g[y]}=6$,
    $g(wv)\in \{1,4\}$, $g(vy)\in \{1,2\}$, $\{5,7\}\subseteq C_g[t]$, $g(vt)\in\{1,2,4,6,8\}$, and $g(wt)\in\{1,3,4,5,7\}$.

    \begin{figure}[htbp]
        \centering
        \begin{minipage}{0.31\linewidth}
            \centering
            \includegraphics[width=0.8\textwidth]{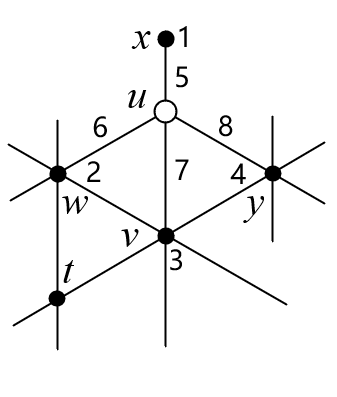}
            \caption{}
            \label{fig:Figure_6}
        \end{minipage}\qquad\qquad\qquad\quad
        \begin{minipage}{0.27\linewidth}
            \centering
            \includegraphics[width=0.9\textwidth]{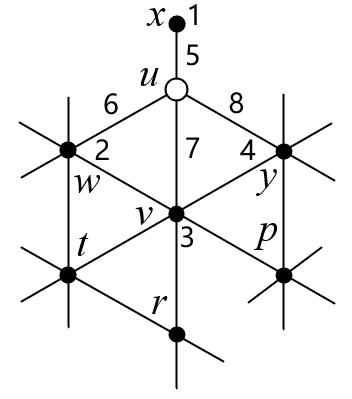}
            \caption{}
            \label{fig:Figure_7}
        \end{minipage}
    \end{figure}

We claim that $\{6,8\}\subseteq C_g[t]$. We show this by considering two cases based on the value of $g(wt)$: $g(wt)\in\{1,3,4\}$ and $g(wt)\in\{5,7\}$. 

First, $g(wt)\in\{1,3,4\}$. If $6\notin C_g[t]$ or $8\notin C_g[t]$, we can recolor $wu$ with $g(wt)$, recolor $wt$ with 6 or 8, and color $u$ with 6. Second, $g(wt)\in\{5,7\}$. In this case, if $6\notin C_g[t]$, we can recolor $uv$ with $g(wv)$, the edge $wv$ with $g(wt)$, $wt$ with 6, $wu$ with 8, $uy$ with 6, and color $u$ with 7; if $8\notin C_g[t]$, we can recolor $uv$ with $g(wv)$, the edge $wv$ with $g(wt)$, $wt$ with 8, and color $u$ with 7. In either case, we can obtain a total 8-coloring of $H$, a contradiction. Therefore, $\{6,8\}\subseteq C_g[t]$, and hence $\{6,8,5,7\}\subseteq C_g[t]$.
    
Note that $g(vt)\in\{1,2,4, 6,8\}$. When $g(vt)\in\{1,2,4\}$, we have $C_g[t]=\{g(vt),5,6,7,8\}$, $g(wt)\in \{5,7\}$, and $g(wv)\notin C_g[t]$. Then, we can obtain a total 8-coloring of $H$ by recoloring $uv$ with $g(wv)$, interchanging the colors of $wv$ and $wt$, and coloring $u$ with 7, a contradiction. When $g(vt)\in\{6,8\}$,  we have $\{6,8,5,7\}\subseteq C_g[t]$ and $\{g(wv),g(vy),5,7,g(vt)\}\subseteq C_g[t]$. Then, $\{g(wv),g(vy),5,7,6,8\}\subseteq C_g[t]$. This shows that $|C_g[t]|>5$,  contradicting to $d_H(t)=4$. This completes the proof of \autoref{lem3-6}(ii).

\textit{Proof of \autoref{lem3-6} (iii).} Suppose to the contrary that $H$ contains a [4,6,6,4,6,6] around $v$, say $[r,t,w,u,y,p]$ around $v$, where $d_H(w)=d_H(y)=d_H(v)=d_H(t)=d_H(p)=6$ and $d_H(u)=d_H(r)=4$ (see \autoref{fig:Figure_7}). From the discussion above, we can see that $\overline{C_g[v]}=5$, $\overline{C_g[w]}=8$, $\overline{C_g[y]}=6$,
    $g(wv)\in \{1,4\}$, $g(vy)\in \{1,2\}$, $\{g(vr),g(vt),g(vp),g(wv),g(vy)\}=\{1,2,4,6,8\}$, $\{5,7\}\subseteq (C_g[t]\cap C_g[r]\cap C_g[p])$, $\{6,8,5,7\}\subseteq C_g[t]$, and $5\in \{g(t),g(r),g(p)\}$. (We can obtain $\{6,8\}\subseteq C_g[t]$ by considering two cases based on the value of $g(wt)$: $g(wt)\in\{1,3,4\}$ and $g(wt)\in\{5,7\}$. It is similar to the situation in \textit{Proof of \autoref{lem3-6} (ii).}) 
    
We claim that $\{6,8,5,7\}\subseteq C_g[p]$. It is sufficient to show that $\{6,8\}\subseteq C_g[p]$. We consider two cases based on the $g(yp)$ value. Note that $g(yp)\in\{1,2,3, 5,7\}$. First,  $g(yp)\in\{1,2,3\}$. If $6\notin C_g[p]$ or $8\notin C_g[p]$, we can recolor $uy$ with $g(yp)$, recolor $yp$ with 6 or 8, and color $u$ with 8. Second, 
$g(yp)\in\{5,7\}$. If $8\notin C_g[p]$, we can recolor $uv$ with $g(vy)$, the edge $vy$ with $g(yp)$, $yp$ with 8, $uy$ with 6, $wu$ with 8, and color $u$ with 7. If $6\notin C_g[p]$, we can recolor $uv$ with $g(vy)$, the edge $vy$ with $g(yp)$, $yp$ with 6, and color $u$ with 7. In either case, we can obtain a total 8-coloring of $H$, a contradiction. Therefore, $\{6,8\}\subseteq C_g[p]$, and hence $\{6,8,5,7\}\subseteq C_g[p]$.

Note that  $g(vr)\in\{1,2,4,6,8\} \setminus \{g(wv),g(vy)\}$. In the following, we deal with two cases based on the   $g(vr)$ value: $g(vr)\in\{6,8\}$, and $g(vr)=\{1,2,4\}\backslash\{g(wv),g(vy)\}$.

\textit{Case 1.} $g(vr)\in\{6,8\}$. In this case, we have $\{g(wv),g(vy),5,7,g(vr)\}\subseteq C_g[r]$ by \autoref{claim4-5}, i.e. $C_g[r]=\{g(wv),g(vy),5,7,g(vr)\}$($d_H(r)=4$). 
    Let $\{k\}=\{6,8\}\setminus\{g(vr)\}$.  It is evident that $k\notin C_g[r]$, and $g(tr)\in\{g(wv),g(vy),5,7\}$.
   Additionally, we will explore two subcases: $g(vt)=k$ and $g(vp)=k$.

\textit{Case 1.1.} $g(vt)=k$. Consider the color assigned to the edge $tr$. If $g(tr)=g(wv)$, we can recolor both $wu$ and $vt$ with $g(wv)$, recolor both $wv$ and $tr$ with $k$, and color $u$ with 6. If $g(tr)=g(vy)$, we can recolor both $uy$ and $vt$ with $g(vy)$, recolor both $vy$ and $tr$ with $k$, and color $u$ with 8. If $g(tr)=5$, we can recolor $wu$ with $g(wv)$, recolor both $wv$ and $tr$ with $k$, recolor $vt$ with 5, and color $u$ with 6.  If $g(tr)=7$, we can recolor $uv$ with $g(wv)$, recolor both $wv$ and $tr$ with $k$, recolor $vt$ with 7, and color the vertex $u$ with 7; in particular, when $k=g(wu)$, we need to interchange the colors of the edges $wu$ and $uy$.

\textit{Case 1.2.} $g(vp)=k$. In this case,  we have $\{g(wv),g(vy),g(vt)\}=\{1,2,4\}$, $\{g(wv),g(vy),5,7,k\}\subseteq C_g[p]$ by \autoref{claim4-5}, and $\{6,8,5,7\}\subseteq C_g[p]$, i.e. $\{g(wv),g(vy),5,7,6,8\}\subseteq C_g[p]$. Hence, $\overline{C_g[p]}\in \{g(vt),3\}$ and $\overline{C_g[r]}=\{g(vt),3,k\}$.

First, $\overline{C_g[p]}=g(vt)$. We consider the color assigned to the edge $tr$. If $g(tr)=g(wv)$, we can recolor both $vp$ and $tr$ with $g(vt)$, recolor both $vt$ and $wu$ with $g(wv)$, recolor $wv$ with $k$, and color $u$ with 6. If $g(tr)=g(vy)$, we can recolor both $vp$ and $tr$ with $g(vt)$, recolor both $vt$ and $uy$ with $g(vy)$, recolor $vy$ with $k$, and color $u$ with 8.  If $g(tr)=5$ or $g(tr)=7$, we can recolor both $uv$ with $g(vt)$, interchange the colors of the edges $vt$ and $tr$, and color the vertex $u$ with 7.

Second, $\overline{C_g[p]}=3$. In this case, we consider the colors assigned to the vertices $t$, $r$, and $p$. If $k\notin \{g(t),g(r),g(p)\}$, we can recolor $v$ with $k$, recolor $vp$ with 3, and color $u$ with 3. If $g(vr)\notin \{g(t),g(r),g(p)\}$, we can recolor $v$ with $g(vr)$, recolor $vr$ with 3, and color $u$ with 3. (As above mentioned, if $5\notin \{g(t),g(r),g(p)\}$, we can recolor the vertex $v$ with 5 and color the vertex $u$ with 3). Consequently, we have  $\{g(t),g(r),g(p)\}=\{k,g(vr),5\}=\{6,8,5\}$, i.e. $7\notin \{g(t),g(r),g(p)\}$.   Then, we can recolor the vertex $v$ with 7, recolor $vp$ with 3, recolor $uv$ with $g(vy)$, recolor $vy$ with $k$, and color the vertex $u$ with 3; we also need to interchange the colors of the edges $wu$ and $uy$ when $k=g(uy)$.

\textit{Case 2.} $g(vr)=\{1,2,4\}\setminus \{g(wv),g(vy)\}$. In this case, we have $g(vt)\in\{6,8\}$. Observe that $\{6,8,5,7\}\subseteq C_g[t]$, and $\{g(wv),g(vy),5,7,g(vt)\}\subseteq C_g[t]$ by \autoref{claim4-5}, i.e. $\{g(wv),g(vy),5,7,6,8\}\subseteq C_g[t]$.
    We can see that $\overline{C_g[t]}\in\{3,g(vr)\}$, $5 \in \{g(t),g(r),g(p)\}$, and $\{5,7,g(vr)\}\subseteq C_g[r]$.
We further consider two cases based on $\overline{C_g[t]}$: $\overline{C_g[t]}=3$ and $\overline{C_g[t]}=g(vr)$.
    
\textit{Case 2.1.} $\overline{C_g[t]}=3$. We first analyze the colors assigned to the vertices $t$, $r$, and $p$.
If $g(vt)\notin \{g(t),g(r),g(p)\}$, we can recolor $v$ with $g(vt)$, recolor $vt$ with 3, and color $u$ with 3. 
If $7\notin\{g(t),g(r),g(p)\}$, we can recolor $v$ with 7, recolor $uv$ with $g(wv)$, recolor $wv$ with $g(vt)$, recolor $vt$ with 3, and color the vertex $u$ with 3; we also need to interchange the colors of the edges $wu$ and $uy$ when $g(vt)=g(wu)$.

The above shows that $\{g(t),g(r),g(p)\}=\{g(vt),7,5\}$, $1\notin \{g(t),g(r),g(p)\}$, and $1\in \{g(wv),g(vy),g(vr)\}$. Furthermore, when 
 $g(wv)=1$, we can recolor both $wu$ and $v$ with 1, recolor $wv$ with $g(vt)$, recolor $vt$ with 3, and color $u$ with 3.
When $g(vy)=1$, we can recolor both $uy$ and $v$ with 1, recolor $vy$ with $g(vt)$, recolor $vt$ with 3, and color $u$ with 3. 
When $g(vr)=1$, if $g(wv)\notin C_g[r]$, we can recolor both $vr$ and $wu$ with $g(wv)$, recolor $wv$ with $g(vt)$, the edge $vt$ with 3, the vertex $v$ with 1, and color $u$ with 3; if $g(vy)\notin C_g[r]$, we can recolor both $vr$ and $uy$ with $g(vy)$, recolor $vy$ with $g(vt)$, $vt$ with 3, $v$ with 1, and color $u$ with 3.

Consequently, we have  $C_g[r]=\{g(wv),g(vy),5,7,g(vr)\}$, and $3\notin C_g[r]$. Then, we recolor $vr$ with 3, recolor $v$ with 1, and color $u$ with 3.

\textit{Case 2.2.} $\overline{C_g[t]}=g(vr)$. If $g(wv)\notin C_g[r]$, we can recolor both $vr$ and $wu$ with $g(wv)$, recolor $wv$ with $g(vt)$, recolor $vt$ with $g(vr)$, and color $u$ with 6.
If $g(vy)\notin C_g[r]$, we can recolor both $vr$ and $uy$ with $g(vy)$, recolor $vy$ with $g(vt)$, recolor $vt$ with $g(vr)$, and color $u$ with 8. Therefore, we have $C_g[r]=\{g(wv),g(vy),g(vr),5,7\}=\{1,2,4,5,7\}$. 
It is evident that $6\notin C_g[r]$, $8\notin C_g[r]$, and $g(tr)\in\{g(wv),g(vy),5,7\}$. We proceed by discussing the color assigned to the edge $tr$.
    
First,  $g(tr)=7$. We can recolor $uv$ with $g(vy)$, recolor both $vy$ and $tr$ with $g(vt)$, recolor $vt$ with 7, and color the vertex $u$ with 7; we also need to interchange the colors of the edges $wu$ and $uy$ when $g(vt)=g(uy)$.
Second, $g(tr)=g(wv)$. Then, we can recolor both $wu$ and $vt$ with $g(wv)$, recolor both $wv$ and $tr$ with $g(vt)$, and color $u$ with 6. Third, 
$g(tr)=g(vy)$. We can recolor both $uy$ and $vt$ with $g(vy)$, recolor both $vy$ and $tr$ with $g(vt)$, and color $u$ with 8. Finally,  $g(tr)=5$. we can recolor $wu$ with $g(wv)$, recolor both $wv$ and $tr$ with $g(vt)$, recolor $vt$ with 5, and color $u$ with 6.

In all the analyzed cases, it has been consistently demonstrated that the graph \( H \) possesses a total 8-coloring in these specific situations. This observation ultimately leads to a contradiction. Consequently, the validity of \autoref{lem3-6} is confirmed. 


\section{Discussion}\label{sec:discussion}
Let \( G \) be a planar graph with a maximum degree of six. If \( G \) contains a subgraph that is isomorphic to a mushroom (as illustrated in \autoref{fig:mushroom}),  a tent (as depicted in \autoref{fig:tent}) or a cone (as shown in \autoref{fig:cone}), we cannot assert, according to \autoref{thm:theorem1}, that \( G \) is totally 8-colorable. As a result, this situation remains an unresolved aspect of TCC. We have tried to expand the methods used in this paper to address this unresolved issue. We rewrite Euler's formula as follows:

\[
\sum_{v \in V(H)} (\lambda d_H(v) - 1) + \sum_{f \in F(H)} (\mu d_H(f) - 1) = -2 < 0,
\]
where \( \lambda \) and \( \mu \) are real parameters such that \( \lambda + \mu = \frac{1}{2} \). When \( \lambda = \mu = \frac{1}{4} \), it is just the formula utilized in this paper. Our goal is to identify the optimal values of \( \lambda \) and \( \mu \) to cope with more complicated cases. Based on these values, we will assign initial charges to a minimal counterexample and devise a discharging procedure aimed at establishing a contradiction. However, this presents a particularly challenging case.

Consider a minimal counterexample in which all \( k \)-neighbors, for \( k = 5, 6 \), of the 5-vertices—each incident with five 3-faces—are also incident with \( k \) 3-faces. To establish a contradiction, it is essential to ensure that these  5-vertices maintain nonnegative charges following the transfer of charges. In particular, the parameters \( \lambda \) and \( \mu \) must satisfy the following inequality:

\[
(\mu \cdot 3 - 1) + (\frac{\lambda \cdot 5 - 1}{5}\cdot 3) \geq 0.
\]
Moreover, consider the following expression:

\[
(\mu \cdot 3 - 1) + \frac{\lambda \cdot (\Delta - 1) - 1}{\Delta - 1}\cdot 3 \geq 0.
\]
From this, one can derive that \( \Delta \geq 7 \). However, when \( \Delta = 6 \), the following holds true:

\[
(\mu \cdot 3 - 1) + (\frac{\lambda \cdot 5 - 1}{5}\cdot 3) = (\mu + \lambda)\cdot 3 - 1 - \frac{3}{5} = -\frac{1}{10} < 0.
\]
This suggests that suitable values for \( \lambda \) and \( \mu \) do not exist under the given conditions. Therefore, when utilizing this method, minimal counterexamples must either ensure that the 5-vertices receive sufficient charges from their neighboring vertices or exclude specific configurations from consideration. In the first scenario, it may be necessary to expand the neighborhoods layer by layer, which increases complexity and calls for the development of more refined discharging rules. In the second scenario, there are numerous combinations to evaluate, and complex situations may require the application of alternative innovative methods or advanced solutions to address them effectively.

\bibliographystyle{ieeetr}

\end{document}